\documentclass[10pt]{article}

\usepackage{amsmath,amssymb,amstext} 
\usepackage[pdftex]{graphicx} 
\usepackage[pdftex,pagebackref=false]{hyperref} 
\usepackage{amsfonts,amsthm,amsmath}
\usepackage{amssymb}
\usepackage{authblk}
\usepackage{bbding}
\usepackage{url}
\usepackage{setspace}
\usepackage{subcaption}
\usepackage{xcolor}
\usepackage{soul}
\usepackage{cite}

\newcommand{\zed}{\ensuremath{\mathbb Z}}
\newcommand{\eff}{\ensuremath{\mathbb F}}

\newtheorem{Theorem}{Theorem}[section]
\newtheorem{Definition}{Definition}[section]
\newtheorem{Example}{Example}[section]
\newtheorem{Remark}{Remark}[section]
\newtheorem{Lemma}[Theorem]{Lemma}
\newtheorem{Corollary}[Theorem]{Corollary}

\newcommand{\xx}{{\ensuremath{\mathbf{x}}}} 
\newcommand{\yy}{{\ensuremath{\mathbf{y}}}} 

\usepackage[top=1.25in, bottom=1.5in, left=1.5in, right=1.5in]{geometry}

\title{Rectangular, Range, and Restricted AONTs: Three Generalizations of All-or-Nothing Transforms}

\author{Navid Nasr Esfahani}
\author{Douglas R.\ Stinson\thanks{D.R.\ Stinson's research is supported by  NSERC discovery grant RGPIN-03882.
}}
\affil{David R.\ Cheriton School of Computer Science\\University of Waterloo\\
Waterloo, Ontario, N2L 3G1\\Canada}

\date{\today}

\begin{document}
\maketitle
\begin{abstract}
	All-or-nothing transforms (AONTs) were originally defined by Rivest \cite{R} as bijections from $s$ input blocks to $s$ output blocks such that no information can be obtained about any input block in the absence of any output block. Numerous generalizations and extensions of all-or-nothing transforms have been discussed in recent years, many of which are motivated by diverse applications in cryptography, information security, secure distributed storage, etc. 
In particular, $t$-AONTs, in which no information can be obtained about any $t$ input blocks in the absence of any $t$ output blocks, have received considerable study.
	
In this paper, we study three generalizations of AONTs that are motivated by applications due to Pham et al.\ \cite{PSC2019Optical} and Oliveira et al.\ \cite{OLVBM}. We term these generalizations rectangular, range, and restricted AONTs. Briefly, in a rectangular AONT, the number of outputs is greater than the number of inputs. A range AONT satisfies the $t$-AONT property for a range of consecutive values of $t$. Finally, in a restricted AONT, the unknown outputs are assumed to occur within a specified set of ``secure'' output blocks. We   
study existence and non-existence and provide examples and constructions for these generalizations. We also demonstrate interesting connections with combinatorial structures such as orthogonal arrays, split orthogonal arrays, MDS codes and difference matrices.
\end{abstract}

\section{Introduction}
\label{intro.sec}

Rivest \cite{R} defined all-or-nothing transforms in the setting of computational security as a mode of operation for block ciphers that can impede brute-force attacks. Stinson \cite{St} introduced and studied unconditionally secure all-or-nothing transforms, i.e., all-or-nothing transforms in the information-theoretic setting. Various generalizations of these transforms have been studied in recent years, including the following:
\begin{itemize}
\item almost AONTs (see \cite{ES17,DES,ZZWG}),
\item $t$-AONTS (see \cite{EGS,ES21_OSAONT,Linear2pp}), and
\item asymmetric AONTS  (see \cite{BastionAONT,ES21_AsymAONT}).
\end{itemize}

In this paper, we study three new types of AONTs motivated by applications due to Pham et al.~\cite{PSC2019Optical} and  Oliveira et al. \cite{OLVBM}. After introducing each of the generalizations, we study existence and non-existence, and provide examples and constructions.\footnote{These generalizations were first formally defined in the PhD thesis of the first author \cite{Thesis}.} We also demonstrate interesting connections with combinatorial structures such as orthogonal arrays, split orthogonal arrays, MDS codes and difference matrices.

We base all the generalizations in this paper on $(t,s,v)$-all-or-nothing transforms \cite{DES}, which are defined informally as follows. 

\begin{Definition}
	\label{def1}
	Suppose $s$ is a positive integer and $\phi : \Gamma^s \rightarrow \Gamma^s$, where $\Gamma$ is a finite set of size $v$ (called an \emph{alphabet}). Thus $\phi$ is a function that maps an input $s$-tuple 
	$\xx = (x_1, \dots  , x_s)$ to an
	output $s$-tuple 
	$\yy = (y_1, \dots  , y_s)$. 
	Suppose $t$ is an integer such that $1 \leq t \le s$.
	
	The function $\phi$ is a \emph{$(t,s,v)$-all-or-nothing transform} (or a $(t,s,v)$-AONT) provided that the following properties are
	satisfied:
	\begin{enumerate}
		\item  $\phi$ is a bijection.
		\item  If any $s - t$ of the $s$ outputs $y_1, \dots , y_s$ are fixed, then the values of any $t$ inputs  $x_1, \dots , x_s$ are completely undetermined.
	\end{enumerate}
\end{Definition}

It is convenient to define an all-or-nothing transform as a certain combinatorial structure. We recall the relevant combinatorial definitions (e.g., see \cite{EGS}) and then we briefly review the security provided by these combinatorial structures when they are used as AONTs. 

First, we require some preliminary definitions. An \emph{$(N,k,v)$-array} is an $N$ by $k$ array, say $A$, whose entries are elements chosen from an alphabet $\Gamma$ of order $v$.  
Suppose the columns of $A$ are
labeled by the elements in the set $C$.  
Let $D \subseteq C$, and define $A_D$ to be the array obtained from $A$
by deleting all the columns $c \notin D$.
We say that $A$ is
\emph{unbiased} with respect to $D$ if the rows of
$A_D$ contain every $|D|$-tuple of elements of $\Gamma$ 
exactly $N / v^{|D|}$ times.

We record the following lemma for future use.

\begin{Lemma}
\label{subset.lem}
Suppose that $A$ is an $(N,k,v)$-array that is unbiased with respect to the set (of columns)
$D$. Then $A$ is unbiased with respect to $D'$ whenever $D' \subseteq D$.
\end{Lemma}

Here is our combinatorial definition of an AONT.

\begin{Definition}
	\label{defunbiased}
	A \emph{$(t,s,v)$-all-or-nothing transform} is a $(v^s,2s,v)$-array, say $A$, with columns 
	labeled $1, \dots , 2s$, 
	that is unbiased with respect to the following subsets of columns:
	\begin{enumerate}
		\item $\{1, \dots , s\}$,
		\item $\{s+1, \dots , 2s\}$, and
		\item $I \cup J$, 
		for all $I \subseteq \{1,\dots , s\}$ with $|I| = t$ and all 
		$J \subseteq \{s+1,\dots , 2s\}$ with $|J| = s-t$.
	\end{enumerate}
\end{Definition}

We observe that a $(t,s,v)$-all-or-nothing transform $\phi$ corresponds to a $(v^s,2s,v)$-array $A$ in an obvious way.
For every input $s$-tuple $\xx = (x_1, \dots  , x_s)$, we create a row of $A$ consisting of
the $2s$ entries
\[ x_1, \dots  , x_s, y_1, \dots  , y_s,\]
where $(y_1, \dots  , y_s) = \phi(x_1, \dots  , x_s)$.
We call   $A$ the \emph{array representation} of the AONT $\phi$.

Let $\phi$ be a $(t,s,v)$-all-or-nothing transform and let $A$ be its array representation. Properties 1 and 2 of Definition \ref{defunbiased} say that $\phi$ is a bijection. Property 3 ensures that, if any $s-t$ outputs are fixed, then any $t$ inputs are undetermined.

The security properties of AONTs satisfying Definition \ref{defunbiased} are investigated in \cite{ES21_OSAONT}
from the standpoint of information theory. We assume an \textit{a priori} distribution on the $v^s$ possible input $s$-tuples such that every input occurs with positive probability. It is shown in \cite{ES21_OSAONT} that 
an AONT satisfying Definition \ref{defunbiased} has the property that any $t$ inputs take on any possible values
with positive probability, given the values of any $s-t$ outputs (this is termed \emph{weak security}). Furthermore, it is proven in \cite{ES21_OSAONT}
that the \textit{a posteriori} information about any $t$ inputs (given the values of any $s-t$ outputs) is equal to the \textit{a priori} information about the $t$ specified inputs if the input $s$-tuples are equiprobable (this is termed \emph{strong security}).

In the remainder of this paper, we will implicitly be treating AONTs as combinatorial objects that satisfy Definition \ref{defunbiased}.


The following two results are immediate consequences of Definition \ref{defunbiased}.

\begin{Theorem}
	\label{inverse1.cor}\cite[Theorem 2.25]{Linear2pp}
	A mapping $\phi: \Gamma^s\to \Gamma^s$ is a $(t,s,v)$-AONT if and only if $\phi^{-1}$ is an $(s-t,s,v)$-AONT.
\end{Theorem}

\begin{proof}
Interchange the first $s$ and the last $s$ columns of the array representation of $\phi$.
\end{proof}

Our second result is an existence result phrased in terms of orthogonal arrays. 
An \emph{orthogonal array} OA$(s,k,v)$ is a $(v^s,k,v)$-array that is unbiased with respect to any subset of $s$ columns. 

\begin{Theorem}
	\label{OAthenAONT}\cite[Corollary 35]{DES}
	An OA$(s,2s,v)$ is a $(t,s,v)$-AONT
	for all $t$ such that $1 \leq t \leq s$.
\end{Theorem}

Suppose $q$ is a prime power and the alphabet is $\eff_q$. If every output of a $(t,s,v)$-AONT is an $\eff_q$-linear function of the inputs,  the AONT is a \emph{linear} $(t,s,q)$-AONT. We will write a linear 
$(t,s,q)$-AONT in the form $\yy =  \xx M^{-1}$, where $M$ is an invertible $s$ by $s$ matrix over $\eff_q$ (as always, $\xx$ is an input $s$-tuple and $\yy$ is an output $s$-tuple). Of course this is equivalent to saying  that $\xx = \yy M$.

\begin{Theorem}
\label{submatrix.thm}
	\cite[Lemma 1]{DES}
	Suppose $q$ is prime power and $M$ is an invertible $s$ by $s$ matrix with entries from $\eff_{q}$. Then $\yy =  \xx M^{-1}$ defines a linear $(t,s,q)$-AONT if and only if all $t$ by $t$ submatrices of $M$ are invertible.
\end{Theorem}

The next result is an immediate consequence of  Theorem \ref{inverse1.cor}.

\begin{Corollary}
	\label{inverse2.cor}\cite[Theorem 2.26]{Linear2pp}
	Suppose that $\yy = \xx M^{-1}$ defines a linear $(t,s,q)$-AONT. Then $\yy  = \xx M$ defines a linear $(s-t,s,q)$-AONT.
\end{Corollary}

Now, from Corollary \ref{inverse2.cor} and  Theorem \ref{submatrix.thm}, we obtain the following.

\begin{Corollary}
	\label{linear-st}\cite{Linear2pp}
	Suppose $M$ is an  invertible $s$ by $s$ matrix with entries from $\eff_q$. Then $\yy  = \xx M$ 
	defines a linear $(t,s,q)$-AONT if and only if every $(s-t)$ by $(s-t)$ submatrix of $M$ is invertible.
\end{Corollary}


In the rest of this section, we will briefly discuss two applications that motivate our three generalizations of AONTs. 




 Two of the AONT generalizations discussed in this paper are motivated by the work by Oliveira et al.~\cite{OLVBM}, where they considered both the confidentiality and the availability of  information distributed and stored on a cloud. More specifically, they studied linear erasure codes that can encode an $s$-tuple $X\in \eff^s_{q}$ to an $(s+n)$-tuple $Y\in \eff^{s+n}_{q}$, such that any $s$ symbols from $Y$ can be used to reconstruct $X$. Furthermore, for a positive integer $t
  \leq s$, no information can be obtained about any $t$ symbols of $X$ in the absence of any $n+t$ symbols of $Y$ (this is an ``AONT-like'' property.) This motivates our definition of a \emph{rectangular AONT} that we give in Section \ref{sec:RecAONT}.
 
The paper \cite{OLVBM} constructed the desired codes using a generator matrix that is an $s$ by $s+n$ super-regular matrix.\footnote{A matrix is \emph{super-regular} if all its square submatrices are invertible. The authors of  \cite{OLVBM} do not require the matrix entries to be nonzero, but we consider the case where all the $1$ by $1$ submatrices are invertible, i.e., the matrix entries are nonzero.} As indicated in \cite{OLVBM}, Cauchy matrices can be modified to obtain the desired generator matrices.

One consequence of this Cauchy matrix construction method is that the above-mentioned AONT-like property is satisfied for arbitrary values of $t$. In fact, Cauchy matrices provide bijections that are simultaneously $t$-AONTs for all possible relevant values of $t$, a fact that was noted explicitly
in \cite[Theorem 6]{EGS}. This motivates our definition of \emph{range AONTs} (which include the special case of \emph{strong AONTs}) that we give in Section \ref{sec:rangeAONT}.

The second motivating application is due to Pham et al.~\cite{PSC2019Optical}, who studied the use of all-or-nothing transforms in the secure transmission of information across two channels, where one of the channels is using optical encryption to provide security. Their results are valid if the secure channel is achieved using another information theoretically secure scheme, for example, a one-time pad.

In this scenario, the message is broken into input blocks. Output blocks can be computed by applying the transform on the input blocks. Finally, the output blocks are divided into two disjoint subsets, where one of the subsets is of size $t$. The blocks in the $t$-subset are sent over the secure channel, while the other blocks are communicated via a public channel. Since we know which output blocks are transmitted over the secure channel, the all-or-nothing transform only needs to satisfy a weaker condition, namely that no information can be obtained about any input block as long as the output blocks that are sent over the secure channel are not available. Consequently, Pham et al.~\cite{PSC2019Optical} define \emph{restricted AONTs} so that they satisfy this condition. 
We investigate these AONTs further in Section \ref{sec:restAONT}.

\section{Rectangular AONTs}
\label{sec:RecAONT}

We formally define rectangular AONTs as follows.

\begin{Definition}
	\label{unbiased-rec.def}
	Suppose $s$, $n$, and $t$ are  positive integers, where $t \le s \le n$. 	
	A \emph{$(t,s,n,v)$-recAONT} is a $(v^s,s+n,v)$ array, with columns labeled $1, \dots, s+n$, that is unbiased with respect to the following sets of columns:
	\begin{enumerate}
		\item $\{1, \dots , s\}$
		\item any $J \subseteq \{s+1, \dots , s+n\}$ where $|J| = s$
		\item $I \cup J$, for any sets $I$ and $J$ where $I \subseteq \{1,\dots, s\}$, $|I| = t$,  $J \subseteq \{s+1, \dots , s+n\}$ 
		and $|J| = s-t$.
	\end{enumerate}
	When $n=s$, we have a $(t,s,v)$-AONT. 
\end{Definition}

The following result is a straightforward generalization of Theorem \ref{OAthenAONT}.

\begin{Theorem}
	\label{OAtorecAONT}
	An OA$(s,s+n,v)$, where $n \geq s$, is a $(t,s,n,v)$-recAONT for all $t$, $1 \leq t \leq s$.
\end{Theorem}

We now use Theorem \ref{OAtorecAONT} to give some interesting examples of recAONTs.
It is well-known that an OA$(2,k,v)$ is equivalent to a set of $k-2$ mutually orthogonal Latin squares (MOLS) of order $v$ (see, e.g., \cite{StBookCD}). Many results on MOLS can be found in the \emph{Handbook of Combinatorial Designs} \cite{CD}. These results also provide constructions of recAONTs with $s = 2$ for alphabet sizes that are not required to be a prime power.

For example, suppose we consider $k = 5$. It is well-known that an OA$(2,5,v)$ exists for all $v \geq 4$, $v \neq 6,10$ (see \cite[p.~126]{CD}). 
Hence, we have the following existence result for recAONT.

\begin{Corollary}
	Suppose $v \geq 4$, $v \neq 6,10$.
	Then there exists a $(t,2,3,v)$-recAONT for $t =1,2$.
\end{Corollary}

We now observe that OA$(2,k,v)$ are equivalent to certain recAONT.

\begin{Theorem}
	An OA$(2,k,v)$ is equivalent to a $(1,2,k-2,v)$-recAONT.
\end{Theorem}

\begin{proof}
	Applying Theorem \ref{OAtorecAONT} with $s=2$, $t=1$, it follows that existence of an OA$(2,k,v)$ implies the existence of a $(1,2,k-2,v)$-recAONT. For the converse, we observe that the array representation of a $(1,2,k-2,v)$-recAONT is unbiased with respect to any two columns, and hence it is also the array representation of an OA$(2,k,v)$.
\end{proof}


We now discuss a connection between recAONT and split orthogonal arrays, which are structures defined by Levenshtein \cite{Leven}. A \emph{split orthogonal array} SOA$(t_1,t_2;s_1,s_2;v)$ is a
$(v^{t_1+t_2}, s_1+s_2,v)$-array $A$ that satisfies the following two properties:
\begin{enumerate}
\item the columns of $A$ are partitioned into two sets, of sizes $s_1$ and $s_2$, and 
\item $A$ is unbiased with respect to any set of $t_1 + t_2$ columns, where $t_1$ columns are chosen from the first set of columns and $t_2$ columns are chosen from the second set of columns. 
\end{enumerate}

The following result due to Bill Martin (private communication) is a straightforward consequence of Definition \ref{unbiased-rec.def}.

\begin{Theorem}
	Suppose there exists a $(t,s,n,v)$-recAONT.  Then there exists an SOA$(t,s-t;s,n; v)$.
\end{Theorem}

\begin{proof}
	From Theorem \ref{unbiased-rec.def} we know that a $(t,s,n,v)$-recAONT is equivalent to a $(v^s, s+n, v)$-array, that is unbiased with respect to $I\cup J$, for any sets $I$ and $J$ where $I \subseteq \{1,\dots, s\}$, $|I| = t$,  $J \subseteq \{s+1, \dots , s+n\}$ 
	and $|J| = s-t$. If we set $n_1 = s, n_2 = n, t_1 = t$, and $t_2 = s-t$, then from the definition of split orthogonal arrays, such an array is an SOA$(t,s-t; s, n; v)$.
\end{proof}

Hence, from a design theoretic perspective, rectangular AONTs are structures ``between'' orthogonal arrays and split orthogonal arrays, in the sense that existence of a suitable orthogonal array implies the existence of a certain recAONT, which in turn implies the existence of a certain split orthogonal array.

Similar to the other types of AONT structures discussed so far, a recAONT is \emph{linear} if its outputs are a linear combination of its inputs. We write a linear recAONT in the form $\yy = \xx N$, where $N$ is an $s$ by $n$ matrix that satisfies certain properties, as given in the following theorem.

\begin{Lemma}
	\label{linearRecAONT}
	Suppose that $q$ is a prime power and $N$ is an $s$ by $n$ matrix with entries from $\eff_q$. Then $\yy =  \xx N$  defines a linear $(t,s,n,q)$-recAONT if and only if the following conditions are satisfied:
	\begin{enumerate}
		\item every $s$ by $s$ submatrix of N is invertible, and
		\item every $(s-t)$ by $(s-t)$ submatrix of $N$ is invertible.
	\end{enumerate}
\end{Lemma}

\begin{proof}
	Clearly property 1 in Definition \ref{unbiased-rec.def} is satisfied if and only if every $s$ by $s$ submatrix of $N$ is invertible.
	We prove that property 2 holds if and only if every $(s-t)$ by $(s-t)$ submatrix of $N$ is invertible.
	
	Let $N'$ be a matrix consisting of any $s$ columns of $N$. Then $\yy' = \xx N'$ is  a $(t,s,v)$-AONT. 
	Therefore, from Corollary \ref{linear-st}, any $(s-t)$ by $(s-t)$ submatrix of $N'$ is invertible. 
\end{proof}


\section{Range and Strong AONTs}
\label{sec:rangeAONT}
In this section, we will study \emph{range AONTs}, where the AONT provides the desired security properties for a continuous range of values for $t$, i.e., for $t_1 \leq t \leq t_2$, for specified integers $t_1$ and $t_2$. In particular, if the range consists of all positive integers not exceeding a given integer $t$, we call the AONT a \emph{strong} AONT. 
Here is the formal definition, which first appeared in \cite{Thesis}.

\begin{Definition}
	\label{unbiased-range.def}
	Suppose $s$, $t_1$, and $t_2$ are positive integers, where $t_1 \le t_2 \le s$. 
	A \emph{$([t_1 , t_2], s,v)$-rangeAONT} is a $(v^s,2s,v)$ array, with columns labeled $1,\dots ,2s$, that is unbiased with respect to the following sets of columns:
	\begin{enumerate}
		\item $\{1, \dots , s\}$,
		\item $\{s+1, \dots , 2s\}$,
		\item $I \cup J$, for any sets $I$ and $J$ where $I \subseteq \{1,\dots, s\}$, $|I| = t$ and $t_1\le t \le t_2$,  $J \subseteq \{s+1, \dots , 2s\}$, and $|J| = s-t$.
	\end{enumerate} 
\end{Definition}

Thus, a $(t,s,v)$-AONT is the exactly the same as a $([t,t], s,v)$-rangeAONT. 
The following lemma is an immediate consequence of the definition.

\begin{Lemma}
\label{range.lem}
Suppose $t_1 \le t_1' \leq t_2' \leq t_2 \le s$. Then a $([t_1 , t_2], s,v)$-rangeAONT 
is also a $([t_1' , t_2'], s,v)$-rangeAONT.
\end{Lemma}

\begin{Definition}
	\label{StrAONT}
	A \emph{$(t,s,v)$-strong AONT} is a  $([1,t],s,v)$-rangeAONT.
\end{Definition}

The following corollary is an immediate consequence of Lemma \ref{range.lem}.

\begin{Corollary}
A \emph{$(t,s,v)$-strong AONT} is a  $([t_1,t_2],s,v)$-rangeAONT if
$1 \leq t_1 \leq t_2 \leq t$.
\end{Corollary}

We should note that $(t,s,v)$-AONTs are not automatically strong. For example, the optimal linear $(2,p,p)$-AONTs (which exist for all primes $p$) constructed in \cite{EGS,Linear2pp} are not strong. This is because the relevant matrices $M$ contain $0$ entries and hence they are not $(1,p,p)$-AONTs.

The next result  follows from Theorem \ref{OAthenAONT}.

\begin{Theorem}
	\label{OAtoStrongAONT}
	An OA$(s,2s,v)$ is an $(s,s,v)$-strong AONT.
\end{Theorem}

Similar to the case of $t$-AONTs, we define a \emph{linear range AONT} as a range AONT such that each output element is a linear function of the input elements. We write a linear range AONT in the form  $\yy =\xx M^{-1}$, where $M$ is an $s$ by $s$ invertible matrix. 

\begin{Theorem}
	Suppose that $q$ is a prime power and $M$ is an invertible $s$ by $s$ matrix with entries from $\eff_q$. Then $\yy =\xx M^{-1}$ is a $([t_1,t_2], s,q)$-rangeAONT if and only if all $t$ by $t$ submatrices of $M$ are invertible, for all $t$ such that $t_1 \le t \le t_2$.
\end{Theorem}

We give some small examples of linear $(2,p,p)$-strong AONTs. The defining matrices are invertible, they have no 0 entries, and all 2 by 2 submatrices are invertible.

\begin{Example}
		\label{range3}
	A linear $(2,2,3)$-strong AONT:
	\[
	\left(\begin{array}{c c }
		1 & 1  \\
		1 & 2 
	\end{array}
	\right).
	\]
\end{Example}

\begin{Example}
		\label{range5}
	A linear $(2,3,5)$-strong AONT:
	\[
	\left(\begin{array}{c c c }
		1 & 1 & 1 \\
		1 & 2 & 3  \\
		1 & 3 & 4  
	\end{array}
	\right).
	\]
\end{Example}

\begin{Example}
		\label{range7}
	A linear $(2,5,7)$-strong AONT:
	\[
	\left(\begin{array}{c c c c c}
		1 & 1 & 1 & 1 & 1 \\
		1 & 2 & 3 & 4 & 5 \\
		1 & 3 & 4 & 5 & 6 \\ 
		1 & 4 & 5 & 6 & 2 \\
		1 & 5 & 6 & 2 & 4
	\end{array}
	\right).
	\]
\end{Example}

\begin{Example}
		\label{range9}
	A linear $(2,6,9)$-strong AONT, where
	$\eff_9 = \zed_3[x] / (x^2 + 1)$:
\[
\left(\begin{array}{c c c c c c}
1 & 1 & 1 & 1 & 1 & 1 \\
1 & 2 & x & x+1 & x+2 & 2x \\
1 & x & 2 & 2x & 2x+1 & x+1 \\
1 & x+1 & 2x+2 & x+2 & 2x & 2x+1 \\
1 & x+2 & 2x & x & 2x+2 & 2 \\
1 & 2x & x+2 & 2 & x+1 & x
\end{array}
\right).
\] 

\end{Example}

An $s$ by $s$ Cauchy matrix over $\eff_q$ exists  if  $q \ge 2s$.  These matrices are constructed as follows. Let $r_1, \dots , r_s$, $c_1, \dots , c_s$ be $2s$ distinct elements of $\eff_q$. Then the matrix $M = (m_{ij})$ defined by $m_{ij} = 1/(r_i - c_j)$ is a Cauchy matrix. 
All square submatrices of an $s\times s$ Cauchy matrix are invertible. Hence, any $s\times s$ Cauchy matrix over $\eff_q$ is an $(s, s,q)$-strong AONT.

For fixed positive integers $t_1,t_2$ with $t_1 \leq t_2$, and any for prime power $q$,
define \[\mathcal{S}_R([t_1,t_2],q) = \{ s: \text{there exists a linear } ([t_1,t_2],s,q)\text{-rangeAONT}  \}.\]

\begin{Lemma}
\label{lower.lem}
Suppose that $q \geq 2t_2$. Then $\lfloor \frac{q}{2}\rfloor \in \mathcal{S}_R([t_1,t_2],q)$
\end{Lemma}

\begin{proof}
Cauchy matrices yield linear $([t_1,t_2], \lfloor \frac{q}{2}\rfloor, q)$-rangeAONTs, for all $t_1$ and $t_2$ such that $1\le t_1\le t_2\le \lfloor \frac{q}{2}\rfloor$.
\end{proof}

\begin{Lemma}
\label{minus1.lem}
If $s \in \mathcal{S}_R([t_1,t_2],q)$ and $s > t_2$, then $s-1 \in \mathcal{S}_R([t_1,t_2],q)$.
\end{Lemma}

\begin{proof}
The proof is identical to \cite[Theorem 20]{EGS}.
\end{proof}

\begin{Lemma}
\label{upper.lem}
If $s \in \mathcal{S}_R([t_1,t_2],q)$  then $s \leq \max \{q+t_1 - 1, t_1 + 1\}$.
\end{Lemma}

\begin{proof}
Suppose $s \in \mathcal{S}_R([t_1,t_2],q)$. Then there exists a $(t_1, s,q)$-AONT. From \cite[Theorem 23]{EGS}, there exists an OA$(t_1,s,v)$. Now apply the Bush bound for orthogonal arrays   (e.g., see \cite[Theorem 24]{EGS}).
\end{proof}

Suppose $q \geq 2t_2$. In view of Lemmas \ref{lower.lem}--\ref{upper.lem}, each set $\mathcal{S}_R([t_1,t_2],q)$ is nonempty and contains a maximum element, which we denote by $M_R([t_1,t_2],q)$. Moreover, $\mathcal{S}_R([t_1,t_2],q)$ contains  all positive integers $s$ such that $t_2\leq s \leq M_R([t_1,t_2],q)$.

The papers \cite{EGS} and \cite{Linear2pp} have studied $M_R([2,2],q)$. We now record some results on $M_R([1,2],q)$ that can be inferred from results in these papers. Note that $M_R([1,2],q)$ is simply the largest value of $s$ such that a linear $(2,s,q)$-strong AONT exists. 

\begin{Theorem}
\label{upperbound.thm}
For any  prime power $q > 2$, $M_R([1,2],q) \leq q-1$.
\end{Theorem}

\begin{proof}
Suppose $q>2$ is  prime power. It is shown in \cite[Theorem 14]{EGS} that $M_R([2,2],q) \leq q$, so it immediately follows that  $M_R([1,2],q) \leq q$. Further, in any $(2,q,q)$-AONT, say $\yy =\xx M^{-1}$,  $M$ contains 0 entries, so the AONT cannot be a $1$-AONT (see \cite[Lemma 14]{EGS}). Hence, $M_R([1,2],q) \leq q-1$.
\end{proof}

We now observe that the upper bound of Theorem \ref{upperbound.thm} can be met whenever $q-1$ is a Mersenne prime.

\begin{Theorem}
\label{mersenne.thm}
Suppose $2^n-1$ is a prime. Then $M_R([1,2],2^n) = 2^n-1$.
\end{Theorem}

\begin{proof}
In \cite[Theorem 11]{EGS}, it is shown that the transformation $\yy =\xx M^{-1}$ is a $(2,2^n-1,2^n)$-AONT if $M$ is a Vandermonde matrix defined over $\eff_{2^n}$ and $2^n-1$ is prime. Since a Vandermonde matrix does not contain 0 entries, this transformation is also a $(1,2^n-1,2^n)$-AONT.
Hence $M_R([1,2],2^n) \geq 2^n-1$. We also have $M_R([1,2],2^n) \leq 2^n-1$ from Theorem \ref{upperbound.thm}.
\end{proof}

We have the following improvement of Theorem \ref{upperbound.thm} when $q >3$ is odd.

\begin{Theorem}
\label{upperbound2.thm}
For any odd prime power $q > 3$, $M_R([1,2],q) \leq q-2$.
\end{Theorem}

\begin{proof} Suppose $q>3$ is an odd prime power. In view of Theorem \ref{upperbound.thm}, we only need to show that a linear $(2, q-1, q)$-strong AONT does not exist. Suppose that $\yy =\xx M^{-1}$ is a $(2,q-1,q)$-strong AONT, where $M$ is a $q-1$ by $q-1$ matrix over $\eff_q$. We can assume that the first row of $M$ consists of $1$ entries. Consider the second and third rows of  $M$ ($M$ has at least three rows because $q >3$). Denote the entries in these rows, from left to right, by $a_1, a_2, \dots , a_{q-1}$  and 
$b_1, b_2, \dots , b_{q-1}$, resp. The $a_i$'s comprise all the nonzero elements of $\eff_q$, as do the $b_i$'s. 

Now, because $q$ is odd, the product of the $a_i$'s is $-1$ and the product of the $b_i$'s is also  $-1$. For $1 \leq i \leq q-1$, define $c_i = a_i /b_i$. Then the product of the $c_i$'s is $1$.  
If the $c_i$'s were all distinct, their product would be $-1 \neq 1$. 
Therefore there exist distinct indices $i$ and $j$ such that $c_i = c_j$. Hence $a_i b_j = a_j b_i$ and the corresponding $2$ by $2$ submatrix 
\[
\left(
\begin{array}{ll}
a_i	&a_j\\
b_i	&b_j 
\end{array}
\right)
\]
of $M$ is not invertible. This is a contradiction.
\end{proof}

\begin{Theorem}
$M_R([1,2],3) = 2$, $M_R([1,2],5) =3$, $M_R([1,2],7) =5$ and $6\leq M_R([1,2],9) \leq 7$.
\end{Theorem}

\begin{proof}
The lower bounds follow from Examples \ref{range3}--\ref{range9}. The upper bounds follow from 
Theorem \ref{upperbound.thm} (for $q=3$) and Theorem \ref{upperbound2.thm} (for $q>3$).
\end{proof}

Some of the above results can be interpreted in terms of difference matrices.
Let $G$ be an abelian group of order $g$, written additively,  let $2 \leq k \leq g$ and let $\lambda \geq 1$. Then a \emph{$(g,k; \lambda)$-difference matrix} is a $k$ by $g \lambda $ matrix $D= (d_{i,j})$ of entries from $G$ such that, for any two distinct rows $i$ and $j$ of $D$, the multiset 
\[ \{d_{i,k} - d_{j,k} : 1 \leq k \leq g \} \]
contains  every  element of $G$ exactly $\lambda $ times.

\begin{Theorem}
\label{equiv.thm}
Suppose $q$ is a prime power. If a  $(2,q-1,q)$-strong AONT exists, then a $(q-1,q-1; 1)$-difference matrix with entries from $\zed_{q-1}$ exists.
\end{Theorem}

\begin{proof}
Suppose  that $\yy =\xx M^{-1}$ is a $(2,q-1,q)$-strong AONT, where $M = (m_{i,j})$ is a $q-1$ by $q-1$ matrix over $\eff_q$. $M$ contains no  0 entries and all of its $2$ by $2$ submatrices are invertible. Fix a primitive element $\alpha \in (\eff_q)^*$.  Every entry  $m_{i,j}$ of $M$ can be written uniquely as  $m_{i,j} = \alpha ^ {d_{i,j}}$, where $d_{i,j} \in \zed_{q-1}$. Define $D = (d_{i,j})$. We claim that $D$ is a $(q-1,q-1; 1)$-difference matrix with entries from $\zed_{q-1}$.

Clearly $D$ has entries from $\zed_{q-1}$. Suppose that $D$ is not a difference matrix. Then there are two distinct rows $i$ and $j$ such that \[ d_{i,k} - d_{j,k} =  d_{i,\ell} - d_{j,\ell} \] for some $k \neq \ell$. Then 
\[   \frac{m_{i,k}}{m_{j,k}} =  \frac{m_{i,\ell} }{m_{j,\ell}} ,\] so the submatrix
\[
\left(
\begin{array}{ll}
m_{i,k}	&m_{i,\ell}\\
m_{j,k}	&m_{j,\ell} 
\end{array}
\right)
\]
of $M$ is not invertible. This is a contradiction.
\end{proof}

We can also prove a partial converse to Theorem \ref{equiv.thm}.

\begin{Theorem}
\label{converse.thm}
Suppose $q$ is a prime power and suppose a $(q-1,q-1; 1)$-difference matrix with entries from $\zed_{q-1}$ exists. Then there is a $q-1$ by $q-1$ matrix $M$ with entries from $\eff_q$, such that all $1$ by $1$ and all $2$ by $2$ submatrices are invertible.
\end{Theorem}

\begin{proof}
Suppose say $D = (d_{i,j})$ is a  $(q-1,q-1; 1)$-difference matrix with entries from $\zed_{q-1}$. Let $\alpha \in (\eff_q)^*$ be a primitive element and define 
$M = (m_{i,j})$ by the rule $m_{i,j} = \alpha ^ {d_{i,j}}$. $M$ is a $q-1$ by $q-1$ matrix over $\eff_q$. Clearly $M$ contains no $0$ entries, so all $1$ by $1$ submatrices are invertible. Suppose a submatrix \[
\left(
\begin{array}{ll}
m_{i,k}	&m_{i,\ell}\\
m_{j,k}	&m_{j,\ell} 
\end{array}
\right)
\]
is not invertible. Then we have
\begin{eqnarray*}
m_{i,k} m_{j,\ell}&=& m_{i,\ell} m_{j,k},\\
d_{i,k} + d_{j,\ell}&=& d_{i,\ell} + d_{j,k}, \text{and}\\
d_{i,k} - d_{j,k} &=& d_{i,\ell}- d_{j,\ell},
\end{eqnarray*}
so $D$ is not a $(q-1,q-1; 1)$-difference matrix. This is a contradiction.
\end{proof}

\begin{Remark}
{\rm
In Theorem \ref{converse.thm}, the matrix $M$ would not yield an AONT unless it is invertible.
}
\end{Remark}

\begin{Remark}
{\rm
In view of Theorem \ref{equiv.thm}, Theorem \ref{upperbound2.thm} can also be derived as a special case of 
\cite[Theorem 1.10]{Drake}, which states that a $(g, 3; 1)$-difference matrix over $\zed_g$ does not exist
if $g$ is even. 
}
\end{Remark}

\begin{Remark}
{\rm The existence or nonexistence of $(2,q-1,q)$-strong AONT is unknown when $q= 2^n$ and $q-1$ is not prime. Unfortunately, there are no currently known results on difference matrices that can help resolve these cases, either positively or negatively. It has been conjectured (e.g., see\cite[Conjecture 5.18, \S V.5.3]{CD}) that there is no 
$(g,g \lambda; \lambda)$-difference matrix over any group whose order is not a prime power, but this conjecture has not been proven. If it were proven, then the nonexistence of the above-mentioned AONTS would follow as a consequence of Theorem \ref{equiv.thm}.
}
\end{Remark}

\begin{Remark}
{\rm Suppose $2^n -1$ is prime. We can give an alternate proof of Theorem \ref{mersenne.thm} by starting with a particular $(2^n -1,2^n -1; 1)$-difference matrix, namely the multiplication table of $\zed_{2^n -1}$, and applying Theorem \ref{equiv.thm}. The resulting matrix $M$, being a Vandermonde matrix, is invertible, so it yields an AONT.
}
\end{Remark}


\section{Restricted AONTs}
\label{sec:restAONT}

Pham et al.\ \cite{PSC2019Optical} introduced  $R$-restricted AONTs. Their definition, restated in the language of unbiased arrays, is as follows: 

\begin{Definition}
Suppose $s$ is a positive integer and $R\subseteq \{1,2, \cdots, s\}$. An $R$-restricted AONT is  a 
$(v^s,2s,v)$ array that is unbiased with respect to the following sets of columns:
\begin{enumerate}
	\item $\{1, \dots , s\}$,
	\item $\{s+1, \dots , 2s\}$,
	\item $\{i\} \cup J$, for any sets $\{i\}$ and $J$ where $i \in \{1,\dots, s\}$, and $J  = \{s+1, \dots , 2s\}\setminus R^{\prime}$, where $R^{\prime}= \{r+s: r\in R\}$. (Note that we add $s$ to each element of $R$ to obtain $R'$, because $R'$ refers  to labels of  columns corresponding to outputs of the AONT.)
\end{enumerate}

\end{Definition}


Pham et al.\ \cite{PSC2019Optical} use these $R$-restricted AONTs in a setting where there is an unconditionally secure communication channel, with a limited bandwidth, as well as a channel that can be observed by the adversary. In this setting, a portion of the message is sent through the secure channel, while the rest is transmitted over the regular one. Pham et al.\ \cite{PSC2019Optical} design the security of their system based on the adversary's lack of access to the portion of the message sent over the secure channel. They wish to guarantee that it is impossible for the adversary to gain any information about any one input block, in the absence of the blocks sent over the secure channel. That is, if the output blocks are all known except for the blocks in $R'$, then no information can be obtained about any specific input block.



The above definition can be generalized and extended in various ways. One possible generalization 
considers the security of any $t$ input blocks, where $t\le |R|$, in the absence of all the output blocks sent over the secure channel. Our generalization is stronger; we consider the security of any $t\le |R|$ input blocks assuming that the adversary can learn all the output blocks except for $t$ of the blocks sent over the secure channel. (Of course, if there are exactly $t$ blocks sent over the secure channel, then the adversary is assumed to have access to none of them, and in this case the two generalizations are equivalent.)

Thus, we propose the following more general definition of an $R$-restricted $(t,s,v)$-AONT. This definition was first given in \cite{Thesis}.

\begin{Definition}
	\label{unbiased-rest.def}
	Suppose $s$ is a positive integer, $R\subseteq \{1,2, \dots, s\}$, and $t$ is an integer such that $1\le t\le |R|$. An \emph{$R$-restricted $(t,s,v)$-AONT} is a $(v^s,2s,v)$ array with columns, labeled $1,\dots, 2s$, that is unbiased with respect to the following sets of columns:
	\begin{enumerate}
		\item $\{1, \dots , s\}$,
		\item $\{s+1, \dots , 2s\}$,
		\item $I \cup J$, for any sets $I$ and $J$ where $I \subseteq \{1,\dots, s\}$, $|I| = t$,  $J \subseteq \{s+1, \dots , 2s\}$, $|J| = s-t$, and $|R^{\prime}\setminus J| = t$, where $R^{\prime}= \{r+s: r\in R\}$.
	\end{enumerate}
\end{Definition}

In other words, in an $R$-restricted $(t,s,v)$-AONT, fixing all the outputs---except for $t$  outputs in $R$---does not yield any information about any $t$ inputs.



The following result is an immediate generalization of Theorem \ref{OAthenAONT}.

\begin{Theorem}
	\label{OAtorestAONT}
	Suppose there exists an OA$(s,2s,v)$. Let $R\subseteq \{1,2,\cdots,s\}$. 
	Then there exists an $R$-restricted $(t,s,v)$-AONT for all $t$, $1 \le t \le |R|$. 
\end{Theorem}

Suppose $R = \{1,2,\dots, \ell\}$, where $\ell \ge t$. The following 
corollary describes the restricted AONT property in the matrix representation of \emph{linear} restricted AONTs.

\begin{Corollary}
	\label{Cor:LinRestAONT}
	Suppose that $q$ is a prime power, $t \le \ell$, and $M$ is an invertible $s$ by $s$ matrix with entries from $\eff_q$. Then the transformation $\yy =\xx M^{-1}$ is a $\{1,2,\dots,\ell\}$-restricted $(t,s,q)$-AONT if and only if all $t$ by $t$ submatrices of $M$ that are contained in the first $\ell$ rows of $M$ are invertible.
\end{Corollary}

This relaxation of conditions allows for restricted AONTs with parameters for which an AONT does not exist. For instance, it was shown in \cite{EGS}  that $(2,6,5)$-AONT and $(2,9,9)$-AONT do not exist. However, Examples \ref{Ex:RestAONT_1} and \ref{Ex:RestAONT_2} present $\{1,2\}$-restricted AONTs for the same values of $s$ and $q$. 

\begin{Example}
	\label{Ex:RestAONT_1}
			A linear $\{1,2\}$-restricted $(2,6,5)$-AONT:
	\[
	\left(
	\begin{array}{cccccc}
	0 & 1 & 1 & 1 & 1 & 1 \\
	1 & 0 & 1 & 2 & 3 & 4 \\
	0 & 0 & 1 & 0 & 0 & 0 \\
	0 & 0 & 0 & 1 & 0 & 0 \\
	0 & 0 & 0 & 0 & 1 & 0 \\
	0 & 0 & 0 & 0 & 0 & 1 \\
	\end{array}
	\right) .
	\]
\end{Example}

\begin{Example}
	\label{Ex:RestAONT_2}
			A linear $\{1,2\}$-restricted $(2,9,9)$-AONT:
	\[
	\left(
	\begin{array}{ccccccccc}
	0 & 1 & 1 & 1 & 1 & 1 & 1 & 1 & 1 \\
	1 & 0 & 1 & \alpha & \alpha^2 & \alpha^3 & \alpha^4 & \alpha^5 & \alpha^6 \\
	0 & 0 & 1 & 0 & 0 & 0 & 0 & 0 & 0 \\
	0 & 0 & 0 & 1 & 0 & 0 & 0 & 0 & 0 \\
	0 & 0 & 0 & 0 & 1 & 0 & 0 & 0 & 0 \\
	0 & 0 & 0 & 0 & 0 & 1 & 0 & 0 & 0 \\
	0 & 0 & 0 & 0 & 0 & 0 & 1 & 0 & 0 \\
	0 & 0 & 0 & 0 & 0 & 0 & 0 & 1 & 0 \\
	0 & 0 & 0 & 0 & 0 & 0 & 0 & 0 & 1 \\
	\end{array}
	\right) .
	\]
	
\end{Example}

When  $\ell = t$, Theorem \ref{Cor:LinRestAONT} requires that any $t$ columns of the $t$ by $s$ submatrix formed by the first $t$ rows of $M$ are linearly independent. To construct such a matrix, we can use the parity check matrix of a maximum distance separable (MDS) code. 
For example, triply extended Reed-Solomon codes can be used to construct $\{1,2,3\}$-restricted $(3,2^n+2,2^n)$-AONTs, as shown in Theorem \ref{Thrm:TriplyExtRS_AONT}. 

\begin{Theorem}
	\label{Thrm:TriplyExtRS_AONT} 
	Let $n$ be a positive integer and let $q = 2^n$. Then a $\{1,2,3\}$-restricted $(3,2^n+2,2^n)$-AONT exists.	
\end{Theorem}

\begin{proof}
	 Let $\omega_1, \omega_2,\dots, \omega_{q-1}$ be distinct elements in the finite field $\eff_{q}$. The following matrix 
	\[
	H=	\left(
	\begin{array}{ccccccc}
	1 & 1 & \cdots & 1 & 1 & 0 & 0 \\
	\omega_1 & \omega_2 & \cdots & \omega_{q-1} & 0 & 1 & 0  \\
	{\omega_1}^2 & {\omega_2}^2 & \cdots & {\omega_{q-1}}^2 & 0 & 0 & 1 \\
	\end{array}
	\right)
	\]
is the parity check matrix of a triply extended Reed-Solomon code over $\eff_{q}$ (see \cite[Ch.\ 11, Theorem 10]{MacSl}). This code has length $q+2$, dimension $q-1$ and minimum distance $4$, so any three columns of $H$ are linearly independent.
	To construct the AONT, we only need to  add $q-1$ additional rows in such a way that the resulting matrix is invertible. This goal can be achieved by choosing rows consisting of a single $1$ entry in column $i$ (for $1 \leq i \leq q-2$) and $0$'s elsewhere.
	The resulting matrix $M$ gives rise to a $\{1,2,3\}$-restricted $(3,q+2,q)$-AONT.
\end{proof}

\begin{Remark}
{\rm 
If we use the dual code of the code used in Theorem \ref{Thrm:TriplyExtRS_AONT}, we can also construct a $\{1,2,\dots, q-1\}$-restricted $(q-1,q+2,q)$-AONT.
}
\end{Remark}

Doubly extended Reed-Solomon codes can be utilized in the construction of $\{1,2,\dots, t\}$-restricted $(t,q+1,q)$-AONTs, as Theorem \ref{Thrm:DoublyExtRS_AONT} states. 

\begin{Theorem}
	\label{Thrm:DoublyExtRS_AONT}
	Let $q$ be a prime power and let $t\le q+1$. Then a $\{1,2,\dots, t\}$-restricted $(t,q+1,q)$-AONT exists.
\end{Theorem}

\begin{proof}
	 For any value of $k$ such that $1 \leq k \leq q+1$, we can construct a doubly extended Reed-Solomon code of length $q+1$, dimension $k$ and distance $q -k+ 2$ (see \cite[Ch.\ 11, Theorem 9]{MacSl}). The parity check matrix of this code can be extended by $k$ rows such that the final matrix is invertible. Since any $q-k+1$ columns of the parity check matrix are linearly independent, the final matrix is a $\{1,2,\dots,t\}$-restricted $(t,q+1,q)$-AONT, where $t = q-k+1$.
\end{proof}

\begin{Remark}
{\rm Example \ref{Ex:RestAONT_1} is an application of Theorem \ref{Thrm:DoublyExtRS_AONT}.
}
\end{Remark}

\section{Conclusion}

In this paper, we have initiated a study of three generalizations and extensions of $(t,s,v)$-all-or-nothing transforms: rectangular, range, and restricted AONTs. It is worth noting that these properties are not necessarily mutually exclusive. An example of this is the combination of strong and rectangular AONTs that are used in the applications described by Oliveira et al.~\cite{OLVBM}. Constructions for most of combinations of these AONT properties have not been studied yet and could result in interesting outcomes both in theory and in application.

\end{document}